\documentclass[reqno]{amsart}
\title{
A note on non-simply connected rational homotopy models}

\author{Syunji Moriya\\
}
\thanks{This work is partially supported by JSPS KAKENHI Grant Number JP17K14192.}
\address{Osaka Central Advanced Mathematical Institute, Osaka Metropolitan University, 3-3-138 Sugimoto, Sumiyoshi-ku Osaka 558-8585 Japan}
\email{moriyasy@gmail.com}
\usepackage{amsmath,amsthm,amssymb,mathrsfs}
\usepackage[matrix]{xy}
\usepackage{here}
\xyoption{all}
\usepackage{enumerate}

\theoremstyle{definition}
\newtheorem{defi}{Definition}[section]

\newtheorem{exa}[defi]{Example}
\newtheorem{rem}[defi]{Remark}
\theoremstyle{plain}

\newtheorem{lem}[defi]{Lemma}

\newtheorem{thm}[defi]{Theorem}

\newcommand{\ZZ}{\mathbb{Z}}
\newcommand{\QQ}{\mathbb{Q}}
\newcommand{\kk}{\mathsf{k}}
\newcommand{\MM}{\mathcal{M}}
\newcommand{\NN}{\mathcal{N}}

\newcommand{\OO}{\mathcal{O}}

\newcommand{\Hom}{\mathrm{Hom}}
\newcommand{\TT}{\mathsf{T}}
\newcommand{\LL}{\mathcal{L}}
\begin{document}
\maketitle

\section{Introduction} G\'omez-Tato--Halperin--Tanr\'e \cite{GHT} proposed a non-simply connected generalization of the  rational homotopy theory of Quillen \cite{Quillen} and Sullivan \cite{Sullivan}. They adopted  local systems of commutative differential graded algebras (CDGA's) as  algebraic models of  spaces and proved that the fiberwise rationalization of a space with arbitrary fundamental group is recovered from the local system. They also established an equivalence between homotopy categories by making use of their minimal models.   As another rational model for a non-simply connected space, a $\pi_1$-equivariant CDGA was introduced by Hain \cite{Hain} under suggestion by Delinge, and studied by Katzarkov-Pantev-To\"en \cite{KPT} and Pridham \cite{Pridham}. This CDGA can be  defined as an algebra of  (polynomial) de Rham forms with coefficients in a large semi-simple local system.  This model is also studied by the present author  \cite{Moriya} in terms of dg-categories with  emphasis on recovering  homotopy invariants.  For example, the rational homotopy groups and Postnikov invariants are algebraically recovered from this equivariant CDGA model under some assumption (see  also \cite{Pridham}). Both of these algebraic models have their own advantages. The local system models  have  clear correspondence with  fibrations over  $K(\pi,1)$-spaces, and the equivariant CDGA models admit a notion of minimal model which is very similar to Sullivan's one. So direct comparison of these models might be important. In this note, we illustrate a part of such  comparison  in a simple example. Concretely speaking, we present an example of local system of CDGA's and compute the corresponding  equivariant CDGA model. Our example  is a non-nilpotent version of an example in \cite{GHT}.  We also  see how some homotopy invariants are recovered from the equivariant model with this example. Especially, we present a relatively small complex which computes the twisted cohomology of the corresponding space. To do so, we make explicit the well-known correspondence between extensions of representations and twisted cohomology classes of degree 1 in terms of dg-categories  for the case of the group $\ZZ\times \ZZ$, based on the general theory developed in \cite{Simpson}.  We also describe a general procedure to obtain the corresponding equivariant CDGA model from a local system model in the case of  $\pi_1=\ZZ\times \ZZ$.

\section{Preliminary}
In this section, we  recall basic notions from other papers. We sometimes omit details, which  can be found   in the references. \\
\indent We suppose  all vector spaces in this note, possibly with some extra structures such as a grading, differential, and product are  defined over a fixed base field $\kk$ of characteristic $0$.  All complexes are non-negatively and  cohomologically graded. The symbol $\otimes=\otimes_{\kk}$ denotes the tensor product over $\kk$. A commutative differential graded algebra (in short, CDGA) is a complex $A$ equipped with a unital, associative and graded commutative multiplication $A\otimes A\to A$ commuting with the differential and satisfying $d(1)=0$. A map of CDGA's is a chain map  which preserves the product and unit.  \\
\indent Formally we will work with the category of simplicial sets and the terms `space' and `fibration' mean simplicial set and Kan fibration, respectively.  Let $K$ be a simplicial set and $C$ a category. A local system  on $K$ in $C$ is a functor from the opposite category of simplexes of $K$ to  $C$, so it consists of a family of objects $F_\sigma$ of $C$ labeled by $\sigma\in \sqcup_{n\geq 0}K_n$ and maps (or morphisms) of $C$ labeled by simplicial maps. A map of local systems is a natural transformation. A local system in the category of CDGA's and their maps  is called shortly a local system of CDGA's.  We use freely the terms about local systems  in \cite{GHT}. To ease a term, by a local system of vector spaces, we mean a locally constant local system in the category of (ungraded) vector spaces and $\kk$-linear maps.  Let $A_K$ ($=A_{PL, K}$ in the notation of \cite{GHT}) be the local system of CDGA's defined by $A_{K,\sigma}=A_{PL}(\Delta^{|\sigma|})$, the CDGA of  polynomial de Rham forms on the standard simplex of the same dimension as $\sigma$ with natural structure maps. As in \cite{GHT}, an $A_K$-algebra is a map $A_K\to B$ of local systems of CDGA's on $K$ such that $B$ is extendable and $H^*(B)$ is locally constant. The authors of \cite{GHT} proposed 1-connected $A_K$-algebras of finite type as  algebraic models of non-simply connected spaces.  For a fibration $p:E\to K$ of simplicial sets, they defined an $A_K$-algebra $\mathcal{F}(E,p)$ by setting $\mathcal{F}(E,p)_\sigma=A_{PL}(E\times _{K}\Delta^{|\sigma|})$ where the pullback is taken for the diagram $E\stackrel{p}{\to}K\stackrel{\sigma}{\leftarrow}\Delta^{|\sigma|}$,  and proved that this $A_K$-algebra recovers the fiberwise rational homotopy type of $E$ if  $K$ is a $K(\pi_1(E),1)$-space and $p$ is 2-connected i.e. $p$ induces bijections on $\pi_0$ and $\pi_1$  (for $\kk=\QQ$ under some finiteness assumption).  We call an $A_K$-algebra which is connected with $\mathcal{F}(E,p)$ by a zigzag of quasi-isomorphisms of $A_K$-algebras (see \cite{GHT}), a {\em local system model of $E$}. \\
\indent We shall recall another algebraic model from \cite{Hain, KPT, Pridham, Moriya}. Let $\pi$ be an abstract group. We say a representation of $\pi$ is semi-simple if it is isomorphic to a direct sum of finite dimensional irreducible representations.  We always regard a tensor of representations as a representation by the diagonal action. A $\pi$-CDGA is a CDGA equipped with a $\kk$-linear  action of $\pi$ which commutes with the differential and product. A map of $\pi$-CDGA's is a map of underlying CDGA's which commutes with the action of $\pi$ and a map of $\pi$-CDGA's is said to be a quasi-isomorphism if it induces an isomorphism on the cohomology. We say a $\pi$-CDGA  is semi-simple if its underlying representation  is semi-simple.  Let $X$ be a pointed connected space. We sometimes identify representations $V$ of $\pi_1(X)$ with  the local systems  of vector spaces on $X$ which associate the vector space $V$ to each simplex and the action of an element of $\pi_1(X)$ to each face map by fixing cycles in 1-skelton of $X$ representing elements of $\pi_1(X)$.  (This identification will be made explicit in the case of $X=T^2$ in next section.)  This identification naturally extends to the equivalence between the category of representations and their  homomorphisms  and the category of local systems of vector spaces and their maps, which sends a homomorphism $f:V\to V'$ to the map given by $f$ on each simplex. Let $V$ be a local system of vector spaces on $X$. For a local system of CDGA's $B$, we define a local system $B\otimes V$ (of complexes) by $(B\otimes V)_\sigma=B_\sigma\otimes V_\sigma$ with the obvious maps, where $V_\sigma$ is regarded as a complex concentrated in degree 0.  We define a complex $C^*_{PL}(X,V)$ by $C^*_{PL}(X,V)=\Gamma(A_{PL,X}\otimes V)$, where $\Gamma$ denotes the global section functor given in  Definition 1.1 of \cite{GHT}.  (This is the same as the complex given in  Definition 3.4.2 of \cite{Moriya}.) While we give the definition of  original equivariant CDGA model $A_{red}(X)$  using an affine group scheme as  basic information in the following, we will not use this definition in  the proofs. Instead, we use a  characterization of equivalent models in terms of dg-categories given in Lemma \ref{Ldga_dgc}. The pro-reductive completion $\pi^{red}$ of an abstract group $\pi$ is the affine group scheme whose finite dimensional representations are in one-to-one correspondence with the finite dimensional semi-simple representations of $\pi$ in a natural manner. Let $\OO(\pi^{red})$ denotes the coordinate ring of $\pi^{red}$ (or Hopf algebra representing $\pi^{red}$).   For a pointed connected space $X$, we regard $\OO(\pi_1(X)^{red})$ as a local system of $X$ by the right translation. We define a $\pi_1(X)$-CDGA
$A_{red}(X)$ as the complex $C^*_{PL}(X,\OO(\pi_1(X)^{red}))$ equipped with the product induced by the products on polynomial forms and the algebra  $\OO(\pi_1(X)^{red})$, and the action of $\pi_1(X)$ induced by the left translation on $\OO(\pi_1(X)^{red})$ (see Definitions 3.3.3 and 3.4.4 of \cite{Moriya}). Since  the translation on $\OO(\pi_1(X)^{red})$ is semi-simple, the action on $A_{red}(X)$ is also semi-simple. By this feature, $A_{red}(X)$ has a  minimal model similar to the original case of Sullivan. There exists a quasi-isomorphism $\NN\to A_{red}(X)$ of $\pi_1(X)$-CDGA such that the underlying CDGA of  $\NN$ is minimal (in the sense of Sullivan, see e.g. \cite{BG} or \cite{FHT})  and  $\NN$ is semi-simple. We call $\NN$ the {\em equivariant minimal model of $X$}.\\
\indent  To extract homotopical information such as the natural action of $\pi_1(X)$ on $\pi_n(X)\otimes_\ZZ \kk$ and twisted cohomology from the equivariant minimal model, we use dg-categories and extensions of their objects. We shall recall some related notions. A dg-category is a category whose sets of morphisms are furnished with structures of complexes for which the identities are cocycles and  the compositions $\Hom_C(V',V'')\otimes \Hom_C(V,V')\to \Hom_C(V,V'')$ are chain maps, see \cite{Simpson} or \cite{Moriya}. The category of 0-th cocycles $Z^0C$ is the $\kk$-linear category whose objects are  the same as those of $C$ and whose morphisms $V\to V'$ are the cocycles of degree $0$ in $\Hom_C(V,V')$. We mean by an isomorphism in $C$  an isomorphism in $Z^0C$.  We say a dg-functor $F:C\to C'$ between dg-categories is a quasi-equivalence if it induces an equivalence $Z^0C\to Z^0C'$ of categories and quasi-isomorphisms between the complexes of morphisms.  An extension in a dg-category $C$ is a pair of maps $V'\stackrel{p}{\to}\tilde V\stackrel{q}{\to} V$ of degree 0 such that $dp=0, dq=0, qp=0$, and a splitting exists. Here, a splitting of $(p,q)$ is a pair of maps $ V'\stackrel{\alpha}{\leftarrow}\tilde V\stackrel{\beta}{\leftarrow} V$ of degree 0 satisfying $\alpha\beta=0, \alpha p=id_{V'}, q\beta=id_V$, and $p\alpha +\beta q=id_{\tilde V}$. We call the class $[\alpha d\beta]\in H^1(\Hom_C(V,V'))$ the extension class of $(p,q)$. 
The following lemma is stated in \cite{Simpson} and the proof is easy and omitted.
\begin{lem}[\cite{Simpson}]\label{Lextension_isom}
Let $C$ be a dg-category and $V'\stackrel{p_i}{\to}V_i\stackrel{q_i}{\to} V$ an extension in $C$ with a splitting $V'\stackrel{\alpha_i}{\leftarrow}V_i\stackrel{\beta_i}{\leftarrow} V$ for $i=1,2$. If  $V_1$ and  $V_2$ have the same extension class, they are isomorphic. An isomorphism $V_1\to V_2$ is given by $p_2\alpha_1+\beta_2q_1-p_2\gamma q_1$ where $\gamma\in \Hom^0_C(V,V')$ is a chain with $d\gamma=\alpha_2d\beta_2-\alpha_1d\beta_1$.\hfill \qedsymbol
\end{lem}
 We shall recall the completion of a dg-category due to Simpson \cite{Simpson}. A Maurer-Cartan (MC-) element on an object $V$ of a dg-cateory $C$ is an element $\eta\in \Hom^1_C(V,V)$ satisfying $d\eta+\eta^2=0$. Let $\bar C$ be a dg-category defined as follows. An object of $\bar C$ is a pair $(V,\eta)$ of an object of $C$ and an MC-element on it. The complex $\Hom_{\bar C}((V,\eta), (V',\eta'))$ of morphisms is the same as $\Hom_C(V,V')$ as a graded vector space, but its differential is given by $d_{\bar C}f=d_Cf+\eta' f-(-1)^{|f|}f\eta$. We identify $C$  with a full sub dg-category of $\bar C$ via the inclusion $V\mapsto (V,0)$.  $\bar C$ is complete, i.e. any class $[\omega]\in H^1(\Hom_{\bar C}((V',\eta'),(V,\eta)))$ comes from an extension. Concretely speaking, an extension which induces the class is given by $(V',\eta')\to \left( V'\oplus V, 
\left[
\begin{array}{cc}
\eta' &  \omega\\
0 & \eta 
\end{array}
\right]\right)\to (V,\eta)$ with an obvious splitting. Let $\widehat{C}\subset \bar C$ be the full sub dg-category spanned by  objects obtained by successive extensions of objects of $C$. The dg-category $\widehat C$ is called the completion of $C$ and  has a  universal property, see \cite{Simpson}.
For a simplicial set $X$, let $T_{PL}(X)$ be the dg-category defined as follows. The objects  of $T_{PL}(X)$ are the finite dimensional local systems of vector spaces on $X$ and the complex of morphisms is given by $\Hom_{T_{PL}(X)}(V,V')=C_{PL}(X,\mathfrak{Hom}(V,V'))$ where $\mathfrak{Hom}(V,V')$ is the natural hom-local system with $\mathfrak{Hom}(V,V')_\sigma=\{\kk\text{-linear maps :}V_\sigma\to V'_\sigma\}$. The composition is the natural one defined by the product of forms and the composition of maps. Let $T^{ss}_{PL}(X)\subset T_{PL}(X)$ be the full sub dg-category spanned by semi-simple local systems. By the universal property, we have quasi-equivalences $\widehat{T^{ss}_{PL}(X)}\simeq \widehat{T_{PL}(X)}\simeq T_{PL}(X)$. Let $B$ be a semi-simple $\pi_1(X)$-CDGA. We define dg-categories $\TT^{ss}B$ and $\TT B$ as follows.  The objects of $\TT^{ss}B$ are the finite dimensional semi-simple representations of $\pi_1(X)$ and $\Hom_{\TT^{ss}B}(V,V')$ is given by the complex of invariants $(B\otimes\mathfrak{Hom}(V,V'))^{\pi_1(X)}$ (for the diagonal action), with  restriction of the standard differential on the tensor product.  Here, $\mathfrak{Hom}(V,V')$ is the hom-representation given by the conjugate action on the space of linear maps $V\to V'$ and regarded as a complex concentrated in degree 0. The composition is defined similarly to $T_{PL}(X)$ using the product on $B$.  We set $\TT B=\widehat{\TT^{ss}B}$. The dg-categories $C=T_{PL}^{ss}(X)$ and  $\TT^{ss}B$ have a natural tensor (dg-)functor $\otimes: C\times C\to C$. For $T_{PL}^{ss}(X)$, it is given by the tensor of local systems on objects, and the combination of the product of forms and the tensor of maps on morphisms $: (\omega_1\cdot f_1, \omega_2\cdot f_2)\mapsto \omega_1\omega_2\cdot f_1\otimes f_2$. For $\TT^{ss}B$, it is defined similarly using the product on $B$. If $H^0(B)=\kk$, since $(\kk\otimes\mathfrak{Hom}(V,V'))^{\pi}$ is naturally bijective to the set of homomorphisms $V\to V'$, there exists a natural equivalence between $Z^0\TT^{ss}B$ and the category of finite dimensional semi-simple representations (and their homomorphisms), which gives the identity between the  object sets. Similarly, there is an obvious equivalence between $Z^0T_{PL}^{ss}(X)$ and the category of finite dimensional semi-simple local systems. By composing these equivalences with  the restriction of  the equivalence between representations and local systems mentioned above, we have the equivalence $Z^0\TT^{ss}B\simeq Z^0T^{ss}_{PL}(X)$.
\begin{lem}[Lemma 3.3.1 and Theorem 3.4.5 of \cite{Moriya}]\label{Ldga_dgc}
Under the above notations, for any semi-simple $\pi_1(X)$-CDGA $B$ which is connected to $A_{red}(X)$ by a zigzag of quasi-isomorphisms between semi-simple $\pi_1(X)$-CDGA's, we have a zigzag of quasi-equivalences 
\[
\TT^{ss}B\simeq \TT^{ss}A_{red}(X)\simeq  T^{ss}_{PL}(X),\quad 
\]
which preserves the tensor functors on the dg-categories and induces  the equivalence $Z^0\TT^{ss}B\simeq Z^0T^{ss}_{PL}(X)$ defined above. This zigzag induces the  zigzag of quasi-equivalences between the completions
\[
\TT B\simeq \TT A_{red}(X)\simeq T_{PL}(X).
\]
\hfill  \qedsymbol
\end{lem}
So, for a given finite dimensional representation $V$ of $\pi_1(X)$ (or local system on $X$), if we obtain the equivariant minimal model $\NN$ of $X$ and an object $(V_0,\eta)$ in $\TT\NN$ corresponding to $V$ through the zigzag in this lemma, we can compute the twisted cohomology $H^*(X,V)$ as the cohomology of $((\NN\otimes V_0)^{\pi_1(X)}, d_\eta=d+\eta)$.  In general, $V_0$ is (isomorphic to) the semi-simplification of $V$. In other words, if $\{0\}=V^0\subset V^1\subset\cdots\subset V^k=V$ is a sequence of $\pi_1(X)$-invariant subspaces such that $V^{i+1}/V^i$ is semi-simple for each $0\leq i\leq k-1$,  $V_0$ is isomorphic to $\oplus_{0\leq i\leq k-1}V^{i+1}/V^i$.  In next section,  we give explicit description of the representation corresponding to  a given cohomology class of degree 1 and  illustrate examples of MC-elements corresponding to  some non-semi-simple representations  in the case of  very simple group $\pi=\ZZ\times \ZZ$.

%%%%%NEW SECTION
\section{MC-elements corresponding to representations of $\ZZ\times \ZZ$}
Throughout the rest of this note,  
we set $\pi=\ZZ\times \ZZ$, and let $g_i$ denote a fixed generator of the $i$-th component of $\pi$ for $i=1,2$. We use the realization of the $K(\pi,1)$-space
\[
T^2=[0,1]^2/\{(0,t_2)\sim (1,t_2), \ (t_1,0)\sim (t_1,1)\}. 
\]
$T^2$ consists of  one 0-cell, two  1-cells $\sigma_1=\{(t_1,0)\}$, $\sigma_2=\{(0,t_2)\}$ and one 2-cell $\tau=\{(t_1,t_2)\}$. (We are formally working with simplicial sets so these cells are subdivided into simplexes implicitly.)
Furthermore, we always identify a representation $V$ of $\pi$ with a  local system $L$ of vector spaces on $T^2$ defined as follows.  We set $L_{\sigma}=V$ for each cell (or simplex) $\sigma$, and the face map $d_0$ corresponding to the inclusion $\{1\}\to[0,1]$ is given by the action of $g_i$ on $\sigma_i$ for $i=1,2$, and the other face map $d_1$ is the identity. We use the same notations for the corresponding local systems as the representations in what follows.
\begin{lem}\label{Lrealization}
Let $(V,r_i)$, $(V',r'_i)$ be two representations of $\pi$, where $r_i\in GL(V)$ and $r'_i\in GL(V')$ denote the actions of $g_i$ for $i=1,2$. Any class of degree 1 in $H^*(C_{PL}(T^2,\mathfrak{Hom}(V,V')))$ is represented by a cocycle of the form $\omega=P(t_1,t_2)dt_1+Q(t_1,t_2)dt_2$ where $P$ and $Q$ are polynomials of  $t_1, t_2$  (not  piecewise but defined on the entire square $[0,1]^2$)  with coeffcients in $\kk$-linear maps $V\to V'$. There exists a unique  polynomial $\mu(t_1,t_2)$  satisfying 
\[
\frac{\partial\mu}{\partial t_1}=P,\quad \frac{\partial\mu}{\partial t_2}=Q,\quad \mu(0,0)=0.
\] ( If $P$ and $Q$ are constant for $t_1, t_2$, clearly we have $\mu=t_1P+t_2Q$.) A representation $(\tilde V, \tilde r_i)$ which has the extension class $[\omega]$ is given by $\tilde V=V'\oplus V$  with $g_i$ acting as 
\[
\tilde r_i=\left[
\begin{array}{cc}
r'_i &-r'_i\mu(2-i,i-1) \\
0 & r_i
\end{array}\right].
\]
A splitting $V'\stackrel{\alpha}{\leftarrow}\tilde V\stackrel{\beta}{\leftarrow}V$ is given by
\[
\alpha=\biggl[ id_{V'}\ \  -\mu(t_1,t_2)\biggr],\qquad \beta=\left[
\begin{array}{c}
\displaystyle \mu(t_1,t_2) \vspace{2mm}\\
id_V
\end{array}
\right].
\]
\end{lem}
\begin{proof}
The claim about the form of a cocycle follows from Proposition A.1.4 of \cite{Moriya}. Since $\omega$ is a cocycle, we have $\frac{\partial P}{\partial t_2}=\frac{\partial Q}{\partial t_1}$. By this equation,  if we put
\[
\begin{split}
\mu(t_1,t_2) &=\int_0^{t_1}P(\tau, t_2)d\tau+\int_0^{t_2}Q(0,\tau)d\tau \\
(&=\int_0^{t_2}Q(t_1, \tau)d\tau+\int_0^{t_1}P(\tau,0)d\tau),
\end{split}
\] $\mu$ satisfies the equations in the claim.  Since $\omega$ is a global section, we have
\[
r'_2P(t_1,1)r_2^{-1}=P(t_1,0),\qquad r'_1P(1,t_2)r_1^{-1}=P(0,t_2),
\]
which imply the equations
\[
r'_2\{\mu(t_1,1)-\mu(0,1)\}r_2^{-1}=\mu(t_1,0),\qquad r'_1\{\mu(1,t_2)-\mu(1,0)\}r_1^{-1}=\mu(0,t_2).
\]
By these equations, we have
\[
\begin{split}
\tilde r_1\tilde r_2&=
\left[
\begin{array}{cc}
r'_1r'_2 & -r_1'r_2'\mu(0,1)-r_1'\mu(1,0)r_2 \\
0 & r_1r_2
\end{array}
\right]\\
&=\left[
\begin{array}{cc}
r'_1r'_2 & -r_1'r_2'\mu(0,1)-r_1'(r_2'(\mu(1,1)-\mu(0,1))r_2^{-1})r_2 \\
0 & r_1r_2
\end{array}
\right]\\
&=\left[
\begin{array}{cc}
r'_1r'_2 & -r_1'r_2'\mu(1,1) \\
0 & r_1r_2
\end{array}
\right]=\tilde r_2\tilde r_1,
\end{split}
\]
so $\tilde V$ is a well-defined representation.
Similarly, we can verify that $\alpha$ is well-defined as follows.
\[
\begin{split}
r'_1(\alpha|_{t_1=1})\tilde r_1^{-1}&=
r'_1\biggl[ id_{V'}\ \ -\mu(1,t_2)\biggr]
\left[
\begin{array}{cc}
(r'_1)^{-1} &\mu(1,0)r_1^{-1} \vspace{2mm}\\
0 & r_1^{-1}
\end{array}\right]\\
&=\biggl[id_{V'}\ \ r'_1\mu(1,0)r_1^{-1}-r'_1\mu(1,t_2)r_1^{-1}\biggr] \\
&=\biggl[id_{V'}\ \ -\mu(0,t_2)\biggr]=\alpha|_{t_1=0}.
\end{split}
\]
The verification on $\beta$ is similar and the rest of conditions on the splitting is easily verified.
\end{proof}

\begin{exa}\label{Eextension}
In the following examples (1), (2), all representations are supposed to have the trivial action of $g_2$, and expressed as the action of $g_1$.
\begin{enumerate}
\item We consider the 2-dimensional representation 
$V_1=\left[\begin{array}{cc} 
c & e \\
0 & c
\end{array}\right]$. By Lemma \ref{Lrealization}, it is easy to see that $\omega=-(e/c)dt_1$ corresponds to the extension $(\kk, c)\to V_1\to (\kk, c)$. By Lemma \ref{Lextension_isom}, $V_1$ is isomorphic to the object 
$V_2=\left(\left[\begin{array}{cc}
c & 0 \\
0 & c
\end{array}\right], \eta=\left[\begin{array}{cc}
0 & -(e/c)dt_1 \\
0 & 0
\end{array}\right]\right)$ in $\widehat{T_{PL}(T^2)}$. An isomorphism $V_1\to V_2$ is given by 
\[
\left[
\begin{array}{cc}
1 & (e/c)t_1 \\
0 & 1
\end{array}
\right].
\]

\item We consider the 3-dimensional representation 
$V_3=\left[\begin{array}{ccc} 
c & e & h \\
0 & c & f \\
0 & 0 & c
\end{array}\right]$.
By Lemma \ref{Lrealization}, a cocycle representing the class of the extension $V_1\to V_3\to (\kk, c)$ is given by
 \[
-\left[
\begin{array}{cc}
c & e \\
0 & c
\end{array}
\right]^{-1}
\left[
\begin{array}{c}
h \\
f
\end{array}\right]
dt_1=
-c^{-2}\left[
\begin{array}{c}
ch-ef \\
cf
\end{array}\right]dt_1.\] The pushforward by the isomorphism in (1) sends this cocycle to 
the cocycle
\[-c^{-2}\cdot {}^t[
ch-(1-t_1)ef, \ 
cf
]dt_1,\] where ${}^t(-)$ denotes the transposition. Since 
${}^t[
t_1(1-t_1), \ 
0
]$ belongs to 
\[
\left(C^*_{PL}\left(T^2,\left[
\begin{array}{cc}
c & 0 \\
0 & c
\end{array}
\right]\right), d_\eta=d+\eta\right)
\] and $d_\eta{}^t[
t_1(1-t_1), \ 
0
]={}^t[1-2t_1,\ 
0]
dt_1$, last cocycle is cohomologous to
\[
-c^{-2}\cdot{}^t[
ch-ef/2, \ 
cf
]dt_1.
\] So $V_3$ is isomorphic to 
\[
V_4=\left(\left[\begin{array}{ccc} 
c & 0 & 0 \\
0 & c & 0 \\
0 & 0 & c
\end{array}\right], -c^{-2}\left[\begin{array}{ccc} 
0 & ce & ch-ef/2 \\
0 & 0 & cf \\
0 & 0 & 0
\end{array}\right]dt_1\right)\] in $\widehat{T_{PL}(T^2)}$.
\item We consider the 3-dimensional representation $V_5$ on which $g_1$ and $g_2$ act by
\[
\left[
\begin{array}{ccc}
c_1 & e_1 & 0 \\
0 & c_1 & 0  \\
0 & 0  & c_1
\end{array}
\right] \text{\  and\  }\left[
\begin{array}{ccc}
c_2 & 0 & e_2 \\
0 & c_2 & 0  \\
0 & 0  & c_2
\end{array}
\right], 
\]
respectively. Similarly to (2), we see that $V_5$ is isomorphic to 
\[
\left(V^{ss}_5, 
\left[
\begin{array}{ccc}
0 & -(e_1/c_1)dt_1 &-(e_2/c_2)dt_2 \\
0 & 0                   & 0                  \\
0 & 0                   & 0        
\end{array}
\right]\right)     
\]
in $\widehat{T_{PL}(T^2)}$ where $V^{ss}_5$ is the 3-dimensional representation on which $g_i$ acts by the scalar multiplication of $c_i$ for $i=1,2$.
\end{enumerate} 
\end{exa}

\section{An example of  local system of CDGA's}
We shall define our main example of $A_{T^2}$-algebra $\LL$ which is a mild, non-nilpotent generalization of Example 6.7 of \cite{GHT}. We set
$(\Lambda Z, d)=(\Lambda (x,y,z, w, u),d)$ with $x,\ y,\ z $ cocycles of degree 3, $w$ a cocycle of degree 6, $\deg u=5,\ du=yz$. We set $\LL_\sigma=A_{T^2,\sigma}\otimes   (\Lambda Z, d)$ for each cell $\sigma$. To define a local system, it is enough to define face maps for the cell structure given in the beginning of previous section.  Let $a_i, b_i$ be invertible elements in $\kk$ for $i=1,2$. Let $d_j$ be the degeneracy map corresponding to the map $\{1-j\}\subset [0,1]$ for $j=0,1$ as before. The map $d_0$ on $\sigma_1$ is given by $d_0(x)=a_1x,\  d_0(y)=b_1y, d_0(z)=a_1z,\ d_0(w)=a_1b_1(w+xy),$ and $d_0(u)=a_1b_1u$. $d_0$ on $\sigma_2$ is given by $d_0(x)=a_2(x+z),\ d_0(y)=b_2y,\ d_0(z)=a_2z,\ d_0(w)=a_2b_2w,$ and $d_0(u)=a_2b_2u$. We set $d_1=id$ on $\sigma_i$ for $i=1,2$. The map $d_{ij}: \LL_\tau\to\LL_{\sigma_i}$  corresponding to the inclusion $[0,1]\to [0,1]^{\times 2}, \ t\mapsto \left\{
\begin{array}{cc}
(t, 1-j) & (i=1) \\
(1-j,t) & (i=2)
\end{array}
\right.$ is given by the following formulas for $i=1,2, j=0,1$. $d_{i0}(y)=b_{3-i}y,\ d_{i0}(z)=a_{3-i}z,\ d_{i0}u=a_{3-i}b_{3-i}u$ for $i=1,2$, and $d_{10}(x)=a_2(x+z),\ d_{20}(x)=a_1x$,  and $d_{10}(w)=a_2b_2(w-t_1yz-dt_1u),\ d_{20}(w)=a_1b_1(w+xy)$, and $d_{i1}=id$  (see Figure \ref{Fface}). This is well-defined as a 1-connected $A_{T^2}$-algebra of finite type  since this is extendable for some subdivision of $T^2$, but this is not an $A$-minimal model in the sense of Definition 3.5 of \cite{GHT}, where we put $A=A_{T^2}$. For our purpose, this local system model works well enough. 
\begin{figure}[H]
\begin{center}
%WinTpicVersion4.32a
{\unitlength 0.1in%
\begin{picture}(32.5500,13.0800)(0.7000,-13.6800)%
% BOX 2 0 3 0 Black White  
% 2 420 225 1425 1231
% 
\special{pn 8}%
\special{pa 420 225}%
\special{pa 1425 225}%
\special{pa 1425 1231}%
\special{pa 420 1231}%
\special{pa 420 225}%
\special{pa 1425 225}%
\special{fp}%
% VECTOR 2 0 3 0 Black White  
% 2 920 631 920 275
% 
\special{pn 8}%
\special{pa 920 631}%
\special{pa 920 275}%
\special{fp}%
\special{sh 1}%
\special{pa 920 275}%
\special{pa 900 342}%
\special{pa 920 328}%
\special{pa 940 342}%
\special{pa 920 275}%
\special{fp}%
% VECTOR 2 0 3 0 Black White  
% 2 920 825 920 1191
% 
\special{pn 8}%
\special{pa 920 825}%
\special{pa 920 1191}%
\special{fp}%
\special{sh 1}%
\special{pa 920 1191}%
\special{pa 940 1124}%
\special{pa 920 1138}%
\special{pa 900 1124}%
\special{pa 920 1191}%
\special{fp}%
% VECTOR 2 0 3 0 Black White  
% 2 1020 725 1385 725
% 
\special{pn 8}%
\special{pa 1020 725}%
\special{pa 1385 725}%
\special{fp}%
\special{sh 1}%
\special{pa 1385 725}%
\special{pa 1318 705}%
\special{pa 1332 725}%
\special{pa 1318 745}%
\special{pa 1385 725}%
\special{fp}%
% VECTOR 2 0 3 0 Black White  
% 2 820 731 465 731
% 
\special{pn 8}%
\special{pa 820 731}%
\special{pa 465 731}%
\special{fp}%
\special{sh 1}%
\special{pa 465 731}%
\special{pa 532 751}%
\special{pa 518 731}%
\special{pa 532 711}%
\special{pa 465 731}%
\special{fp}%
% STR 2 0 3 0 Black White  
% 4 920 675 920 725 5 0 0 0
% $x$
\put(9.2000,-7.2500){\makebox(0,0){$x$}}%
% STR 2 0 3 0 Black White  
% 4 315 681 315 731 5 0 0 0
% $x$
\put(3.1500,-7.3100){\makebox(0,0){$x$}}%
% STR 2 0 3 0 Black White  
% 4 920 1288 920 1338 5 0 0 0
% $x$
\put(9.2000,-13.3800){\makebox(0,0){$x$}}%
% STR 2 0 3 0 Black White  
% 4 915 75 915 125 5 0 0 0
% $a_2(x+z)$
\put(9.1500,-1.2500){\makebox(0,0){$a_2(x+z)$}}%
% STR 2 0 3 0 Black White  
% 4 1570 675 1570 725 5 0 0 0
% $a_1x$
\put(15.7000,-7.2500){\makebox(0,0){$a_1x$}}%
% BOX 2 0 3 0 Black White  
% 2 2320 233 3325 1239
% 
\special{pn 8}%
\special{pa 2320 233}%
\special{pa 3325 233}%
\special{pa 3325 1239}%
\special{pa 2320 1239}%
\special{pa 2320 233}%
\special{pa 3325 233}%
\special{fp}%
% VECTOR 2 0 3 0 Black White  
% 2 2820 639 2820 283
% 
\special{pn 8}%
\special{pa 2820 639}%
\special{pa 2820 283}%
\special{fp}%
\special{sh 1}%
\special{pa 2820 283}%
\special{pa 2800 350}%
\special{pa 2820 336}%
\special{pa 2840 350}%
\special{pa 2820 283}%
\special{fp}%
% VECTOR 2 0 3 0 Black White  
% 2 2820 833 2820 1199
% 
\special{pn 8}%
\special{pa 2820 833}%
\special{pa 2820 1199}%
\special{fp}%
\special{sh 1}%
\special{pa 2820 1199}%
\special{pa 2840 1132}%
\special{pa 2820 1146}%
\special{pa 2800 1132}%
\special{pa 2820 1199}%
\special{fp}%
% VECTOR 2 0 3 0 Black White  
% 2 2920 733 3285 733
% 
\special{pn 8}%
\special{pa 2920 733}%
\special{pa 3285 733}%
\special{fp}%
\special{sh 1}%
\special{pa 3285 733}%
\special{pa 3218 713}%
\special{pa 3232 733}%
\special{pa 3218 753}%
\special{pa 3285 733}%
\special{fp}%
% VECTOR 2 0 3 0 Black White  
% 2 2720 739 2365 739
% 
\special{pn 8}%
\special{pa 2720 739}%
\special{pa 2365 739}%
\special{fp}%
\special{sh 1}%
\special{pa 2365 739}%
\special{pa 2432 759}%
\special{pa 2418 739}%
\special{pa 2432 719}%
\special{pa 2365 739}%
\special{fp}%
% STR 2 0 3 0 Black White  
% 4 2820 683 2820 733 5 0 0 0
% $w$
\put(28.2000,-7.3300){\makebox(0,0){$w$}}%
% STR 2 0 3 0 Black White  
% 4 2215 689 2215 739 5 0 0 0
% $w$
\put(22.1500,-7.3900){\makebox(0,0){$w$}}%
% STR 2 0 3 0 Black White  
% 4 2820 1279 2820 1329 5 0 0 0
% $w$
\put(28.2000,-13.2900){\makebox(0,0){$w$}}%
% STR 2 0 3 0 Black White  
% 4 2815 83 2815 133 5 0 0 0
% $a_2b_2(w-t_1yz-dt_1u)$
\put(28.1500,-1.3300){\makebox(0,0){$a_2b_2(w-t_1yz-dt_1u)$}}%
% STR 2 0 3 0 Black White  
% 4 3745 671 3745 721 5 0 0 0
% $a_1b_1(w+xy)$
\put(37.4500,-7.2100){\makebox(0,0){$a_1b_1(w+xy)$}}%
% VECTOR 2 0 3 0 Black White  
% 2 320 633 320 283
% 
\special{pn 8}%
\special{pa 320 633}%
\special{pa 320 283}%
\special{fp}%
\special{sh 1}%
\special{pa 320 283}%
\special{pa 300 350}%
\special{pa 320 336}%
\special{pa 340 350}%
\special{pa 320 283}%
\special{fp}%
% VECTOR 2 0 3 0 Black White  
% 2 320 838 320 1188
% 
\special{pn 8}%
\special{pa 320 838}%
\special{pa 320 1188}%
\special{fp}%
\special{sh 1}%
\special{pa 320 1188}%
\special{pa 340 1121}%
\special{pa 320 1135}%
\special{pa 300 1121}%
\special{pa 320 1188}%
\special{fp}%
% VECTOR 2 0 3 0 Black White  
% 2 820 1338 475 1338
% 
\special{pn 8}%
\special{pa 820 1338}%
\special{pa 475 1338}%
\special{fp}%
\special{sh 1}%
\special{pa 475 1338}%
\special{pa 542 1358}%
\special{pa 528 1338}%
\special{pa 542 1318}%
\special{pa 475 1338}%
\special{fp}%
% VECTOR 2 0 3 0 Black White  
% 2 1020 1338 1370 1338
% 
\special{pn 8}%
\special{pa 1020 1338}%
\special{pa 1370 1338}%
\special{fp}%
\special{sh 1}%
\special{pa 1370 1338}%
\special{pa 1303 1318}%
\special{pa 1317 1338}%
\special{pa 1303 1358}%
\special{pa 1370 1338}%
\special{fp}%
% STR 2 0 3 0 Black White  
% 4 225 418 225 468 5 0 0 0
% $d_0$
\put(2.2500,-4.6800){\makebox(0,0){$d_0$}}%
% STR 2 0 3 0 Black White  
% 4 220 958 220 1008 5 0 0 0
% $d_1$
\put(2.2000,-10.0800){\makebox(0,0){$d_1$}}%
% STR 2 0 3 0 Black White  
% 4 650 1383 650 1433 5 0 0 0
% $d_1$
\put(6.5000,-14.3300){\makebox(0,0){$d_1$}}%
% STR 2 0 3 0 Black White  
% 4 1195 1383 1195 1433 5 0 0 0
% $d_0$
\put(11.9500,-14.3300){\makebox(0,0){$d_0$}}%
\end{picture}}%
\end{center}
\vspace{2mm}
\caption{Face maps}\label{Fface}
\end{figure}
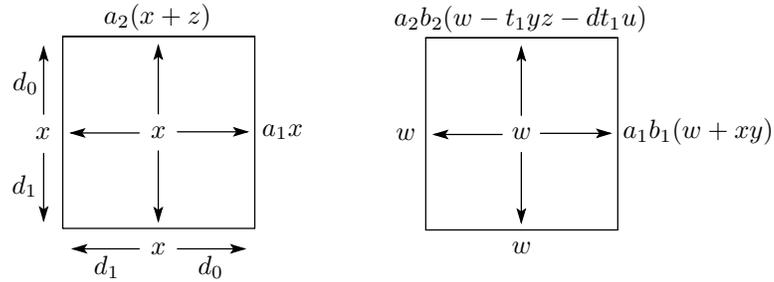
\indent In the rest of the note $X$ denotes the space corresponding to $\LL$ through the equivalence of \cite{GHT}. More precisely, the pointed simplicial set $X$ is the image of the $A_{T^2}$-minimal model of $\LL$ by the realization functor $\langle\ \rangle$ defined there. To compute the equivariant minimal model of $X$, we first give the  equivariant minimal model of $T^2$. The following lemma is a special case of Theorem 4.2.3 of \cite{Moriya}. We give a shorter proof for this case here.
\begin{lem}\label{Lmodel_of_T2}
The equivariant minimal model of $T^2$ is given by the $\pi$-CDGA $\Lambda(s_1,s_2)$ with $s_1, s_2$ cocycles of degree 1 with trivial action. 
\end{lem}
\begin{proof}
Let $\NN$ be the equivariant minimal model of $T^2$. Since the semi-simplicity of representations is preserved by  extension of scalars, we may assume that $\kk$ is algebraically closed. Since commutative matrixes are triangularizable simultaneously, the irreducible representations of $\pi$ are 1-dimensional. By straightforward computation, $H^*(T^2,V)=0$ for a non-trivial irreducible representation $V$ and clearly $H^*(T^2,V)=\Lambda(s_1, s_2)\otimes V$ for a trivial $V$.
This means that the map $(\NN^{\pi}\otimes V)^{\pi}\to (\NN\otimes V)^{\pi} (\simeq C^*_{PL}(T^2,V))$ induced by the inclusion  $\NN^{\pi}\to \NN$ is a quasi-isomorphism for any irreducible $V$, where $\pi$ acts on the subalgebra of invariants $\NN^{\pi}$ trivially. Since $\NN$ is semi-simple, we see that the inclusion $\NN^{\pi}\to \NN$ is a quasi-isomorphism.  If $V$ is a 1-dimensional trivial representation, $\NN^{\pi}$ is isomorphic to the endomorphism algebra $\Hom_{\TT^{ss}\NN}(V,V)$ by definition, which in turn, is quasi-isomorphic to $A_{PL}(T^2)$ as a CDGA by Lemma \ref{Ldga_dgc}.  These observations imply the claim.
\end{proof}
\begin{rem}
In the proof of Lemma \ref{Lmodel_of_T2} and the proof of Theorem \ref{Tmain} below, we apparently use dg-categories. This is mainly for efficiency of notations and we may  write down these proofs using a concrete description of $(\ZZ\times \ZZ)^{red}$ and the standard interpretation of representations as comodules over $\OO((\ZZ\times \ZZ)^{red})$ instead of dg-categories based on the same referred facts.  In the algebraic description of  the cohomology of $X$ with twisted coefficients  given below,  we need dg-categories (or some similar algebraic objects) more essentially since it is based on the latter zigzag of Lemma \ref{Ldga_dgc}.
\end{rem}
We define a semi-simple $\pi$-CDGA $\MM$. We set $\MM=\Lambda(s_1,s_2,\bar x,\bar y,\bar z,\bar w,\bar u)$ with $s_1, s_2$ cocycles of degree 1, $\bar y,\ \bar z$ cocycles of degree 3, $\deg \bar x=3,\ \  d\bar x=s_2\bar z$,\ \ $\deg\bar u=5,\ \  d\bar u=\bar y\bar z$,\ \ $\deg \bar w=6,$ and \ $d\bar w=s_1\bar x\bar y-s_1s_2\bar u$. The action of $\pi$ on $\MM$ is given by $g_i\cdot s_1=s_1$,\ \ $g_i\cdot s_2=s_2$,\ \ $g_i\cdot \bar x=a_i\bar x$,\ \ $g_i\cdot \bar y=b_i\bar y$,\ \ $g_i\cdot \bar z=a_i\bar z$,\ \  $g_i\cdot \bar w=a_ib_i\bar w$,\ and \ $g_i\cdot \bar u=a_ib_i\bar u$ for $i=1,2$. \\
\indent $\pi_i(X)$ has a topologically defined natural action of $\pi=\pi_1(X)$ (see e.g. \cite{Hatcher} or \cite{Whitehead}). We denote by $\pi_i(X)^\vee$ the linear dual  of $\pi_i(X)\otimes_{\ZZ}\kk$ or the dual representation by abuse of notations.
\begin{thm}\label{Tmain}
With the above notations, the equivariant minimal model of $X$ is (isomorphic to) $\MM$.
\end{thm}
\begin{proof}
Let $F$ be a fiber of the fibration $p:X\to T^2$  naturally associated to the map $A_{T^2}\to \LL$. The  topologically defined action of $\pi$ on $\pi_i(X)^\vee$ is identified with the action on the linear dual $\pi_i(F)^\vee$ defined by sliding maps from a sphere along  a  lifted loop using the lifting property of the  fibration (see \cite{Whitehead}). In turn, this action is identified with the action induced by  the map $d_0\circ d_1^{-1}$ on  the indecomposables $H^*(\LL'_{pt}/(\LL'_{pt}\cdot \LL'_{pt}))\cong \kk\langle x, y, z, w, u\rangle$, where $\LL'_{pt}$ is the quotient of the CDGA $\LL_{pt}$ on the base point  of $T^2$ by the submodule $\kk 1$.   So we have  the isomorphisms of representations of $\pi$
\[
(\oplus_i\pi_i(X))^\vee\cong (\oplus_i\pi_i(F))^\vee \cong \kk\langle x, y, z, w, u\rangle.
\]
For example, the action on $\pi_3(X)^\vee$ is identified with the representation 
\[
g_1\mapsto \left[
\begin{array}{ccc}
a_1 & 0& 0 \\
0& b_1 & 0   \\
0&  0    & a_1
\end{array}\right], \quad g_2\mapsto \left[
\begin{array}{ccc}
a_2 & 0 & a_2 \\
0 & b_2 & 0 \\
0 & 0    & a_2
\end{array}\right]
\]
(for the coordinate ${}^t [ z\ y\ x]$)
 through this isomorphism. The semi-simplification of this representation is given by the action of the diagonal part of these matrixes. By this and similar observations for other degree, together with Theorem  4.1.3 of \cite{Moriya} and Lemma \ref{Lmodel_of_T2}, we see that the equivariant minimal model of $X$ has the same number of generators as $\MM$ and the degrees of them and the action of $\pi$ on them are also the same.  Let $L_1$ and $L_2$ be the 1-dimensional representations of $\pi$ on which  each $g_i$ acts by the multiplication of $a_i^{-1}$ and $b_i^{-1}$, respectively. 
Let $S$ be a set of pairs $(k,l)$ of non-negative integers such that the  representations $L^k_1\otimes L^l_2 (=L^{\otimes k}_1\otimes L^{\otimes l}_2)$ with $(k,l)\in S$ appear  exactly once in each isomorphism class of the representations of the form $L_1^{k'}\otimes L_2^{l'}$ with $k',l'\geq 0$. Put 
$U=\oplus_{(k,l)\in S}L_1^k\otimes L_2^l$. We define a commutative product on $U$ by the map 
\[
(L_1^k\otimes L_2^l)\otimes (L_1^{k'}\otimes L_2^{l'})\cong L_1^{k+k'}\otimes L_2^{l+l'}\cong L_1^{k''}\otimes L_2^{l''}
\] where $(k'',l'')\in S$ and the first isomorphism is the transposition and the second is given by $v_1^{k+k'}\otimes v_2^{l+l'}\mapsto v_1^{k''}\otimes v_2^{l''}$ for some fixed generators $v_i\in L_i$. Let $\NN$ be the equivariant minimal model of $X$ and $\LL_{min}$ the $A_{T^2}$-minimal model of $\LL$. We have the following quasi-isomorphisms of semi-simple $\pi$-CDGA's
\[
\begin{split}
\NN\cong (\NN\otimes U)^\pi & \stackrel{\varphi_1}{\longrightarrow} \bigoplus_{(k,l)\in S}C_{PL}^*(X,p^*(L^k_1\otimes L^l_2))
\stackrel{\varphi_2}{\longleftarrow}  \bigoplus_{(k,l)\in S}\Gamma(\mathcal{F}(X,p)\otimes  L^k_1\otimes L^l_2) \\
& \stackrel{\varphi_3}{\longleftarrow} \bigoplus_{(k,l)\in S}
\Gamma(\LL_{min}\otimes  L^k_1\otimes L^l_2)\stackrel{\varphi_4}{\longrightarrow}
 \bigoplus_{(k,l)\in S}
\Gamma(\LL\otimes  L^k_1\otimes L^l_2)\ (=:\tilde \LL).
\end{split}
\]
Here, 
\begin{itemize}
\item the leftmost isomorphism is the obvious one.
\item In the right four objects,  $L_1^k\otimes L_2^l$ is regarded as a local system on $T^2$ and $\Gamma$ denotes the  global section functor.
\item The products on these objects are natural ones. For example, the product on $C^*_{PL}(X,p^*U)$ is the tensor of the product of forms and the product on $U$.
\item  The action on each object is  given by the multiplication of $a_i^kb_i^l$ on the summand tensored with $L_1^k\otimes L_2^l$.
\item $\varphi_1$ is induced by the zigzag of  quasi-equivalences in Lemma \ref{Ldga_dgc}.
\item $\varphi_2$ is the obvious inclusion through the identification $C^*_{PL}(X,  p^*(L_1^k\otimes L_2^l))=\Gamma(A_{PL,X}\otimes p^*(L_1^k\otimes L_2^l))$, which is actually an isomorphism since $X\times_K\Delta^{|\sigma|}$ in the definition of $\mathcal{F}(X,p)$ is simply connected.
\item $\varphi_3$ and $\varphi_4$ are induced by  quasi-isomorphisms of local systems $\LL_{min}\to \mathcal{F}(X,p)$ and $\LL_{min}\to \LL$ respectively, whose existence is ensured  by Theorems 3.12 and 5.10 of \cite{GHT}.
\end{itemize}
The compatibility of  $\varphi_1$  with the product and action is ensured by the compatibility  with the tensor structures mentioned in  Lemma \ref{Ldga_dgc}.
Thus, $\NN$ and last $\pi$-CDGA, denoted by $\tilde \LL$, are quasi-isomorphic.  We shall prove that $\MM$ is quasi-isomorphic to $\tilde \LL$ as a semi-simple $\pi$-CDGA. By abuse of notations,  we denote by $y\in \Gamma(\LL\otimes L_2), z\in \Gamma( \LL\otimes L_1)$ and $u\in \Gamma(\LL\otimes L_1\otimes L_2)$  the constant sections to the representing generator tensored with the fixed generators of $L_i$.  For example, $y_\tau=y\otimes v_2$ for the 2-cell $\tau$. We define an element $x'\in \Gamma(\LL\otimes  L_1)$ of degree 3 by $x'_\tau=(-x+t_2z)\otimes v_1$. Since 
\[
\begin{split}
d_{10}(x'_\tau)& =a_2(-(x+z)+1\cdot z)\otimes a_2^{-1}v_1=-x\otimes v_1=d_{11}(x'_\tau), \text{ and}\\
d_{20}(x'_\tau)& =a_1(-x+t_2z)\otimes a_1^{-1}v_1=d_{21}(x'_\tau),
\end{split}
\]
$x'$ is well-defined as a global section. We also have $d(x'_\tau)=dt_2z\otimes v_1$. We define an element $w'\in\Gamma(\LL\otimes  L_1\otimes L_2)$ of degree 6 by $w'_\tau=(w-t_1xy+t_2dt_1u)\otimes v_1\otimes v_2$. We see that $w'$ is well-defined and $dw'_\tau=(-dt_1 x y-dt_1dt_2u+t_2dt_1zy)\otimes v_1\otimes v_2=(dt_1x'y-dt_1dt_2u)_\tau$ by a straightforward computation. We define a map $F:\MM\to \tilde \LL$ by $s_i\mapsto dt_i$ for $i=1,2$, $\bar \alpha \mapsto \alpha$ for $\alpha=y,z$ and  $u$, $\bar x\mapsto x'$, and $\bar w\mapsto w'$. Clearly, this is a map of $\pi$-CDGA. Since $d\bar x$ is not a boundary in the subalgebra generated by $s_1,s_2,\bar y,\bar z$ and we have already known the generators of the graded algebra $\NN$, we see that $H^{\leq 4}(F)$ is an isomorphism and $H^5(F)$ is a monomorphism. Similarly, as $d\bar u$ and  $d\bar w$ are not boundaries in the subalgebras of generated by generators of lower degrees,  we conclude that $F$ is a quasi-isomorphism.
\end{proof}

Let us look at how some homotopy invariants are recovered from the equivariant minimal model based on the identification in section 3 and the general theory in \cite{Moriya}.\\
\indent {\bf The action of $\pi_1(X)$ on the homotopy groups}. \ Note that the topologically defined action on $\pi_i(X)^\vee$ of $\pi=\pi_1(X)$ is {\em not} compatible with the differential of $\MM$ for any graded linear identification between the generators of $\Lambda Z$ and those of $\MM$ with degree $\geq 3$. Neverthless, we can recover the action on $\pi_i(X)^\vee$.  Let $V^i$ be the space of the generators of $\MM$ of degree $i$ and $\MM(i-1)\subset \MM$ the subalgebra generated by elements of degree $\leq i-1$. As in Theorem 1.0.2 of \cite{Moriya}, the semi-simplification of $\pi_i(X)^\vee$ is isomorphic to $V^i$, and   the extension is encoded in the differential. The differential $d:V^i\to \MM(i-1)^{i+1}\oplus \MM^1\otimes V^i$ followed by the projection to $\MM^1\otimes V^i$ is naturally regarded as an MC-element  $\eta$, and $(V^i,\eta)$ is isomorphic to the object in $\TT\MM$ corresponding to the topological action  on $\pi_i(X)^\vee$ through the equivalence of Lemma \ref{Ldga_dgc} (see sub-subsections 2.2.1 and 3.3.1 of \cite{Moriya} for the sign convention). The MC-element corresponding to the differential on $V^3$ is given by
\[
\left[
\begin{array}{ccc}
0 & 0 & s_2 \\
0 & 0 & 0 \\
0 & 0 & 0
\end{array}
\right].
\]
By Lemma \ref{Lrealization}, this corresponds to the representation
$\left(\left[
\begin{array}{ccc}
a_1 & 0 & 0 \\
0 & b_1 & 0 \\
0 & 0 & a_1
\end{array}
\right], \left[
\begin{array}{ccc}
a_2 & 0 & -a_2 \\
0 & b_2 & 0 \\
0 & 0 & a_2
\end{array}
\right]\right)$. Clearly, this is isomorphic to the action obtained from the original local system in the proof of Theorem \ref{Tmain}. \\ 
\indent {\bf The model in the original rational homotopy theory}.  \ A  CDGA quasi-isomorphic to the usual polynomial de Rham algebra $A_{PL}(X)=\Gamma(A_X)$  is given by $\MM^{\pi}$. So if for any pair of  non-negative integers $(k,l)\not=(0,0)$,   $a_i^kb_i^l\not=1$ holds for at least one of $i=1,2$, we have $A_{PL}(X)\simeq \Lambda(s_1,s_2)$. If $a_i$ is not a root of unity for at least one of $i$, and $b_i=a_i^{-1}$ for both of $i=1,2$, $A_{PL}(X)$ is quasi-isomorphic to the subalgebra of $\MM$ generated by $s_1, s_2, \bar x\bar y, \bar y\bar z, \bar u, \bar w$. \\
\indent {\bf Twisted cohomology}. \  We can obtain an algebraic complex which computes  cohomology of $X$ with coefficients in $V_3$ in Example \ref{Eextension} using the identification given there. Let $V_3^{ss}$ be the semi-simplification of $V_3$ so $g_1$ acts on $V_3^{ss}$ via scalar multiplication by $c$ and $g_2$ does trivially. The object in $\mathsf{T}\MM$ corresponding to $V_3$ is given by simply replacing $dt_1$ in $V_4$ with $s_1$. So $C^*_{PL}(X,V_3)$ is quasi-isomorphic to the complex $((\MM\otimes V_3^{ss})^{\pi}, D)$ with the differential $D$ given by
\[
D(\omega\otimes {}^t[k_1\ k_2\ k_3]
)=d\omega \otimes {}^t[k_1\ k_2\ k_3]
+s_1\omega \otimes M  \, {}^t[k_1\ k_2\ k_3]
\]
where $M=-c^{-2}\left[\begin{array}{ccc} 
0 & ce & ch-ef/2 \\
0 & 0 & cf \\
0 & 0 & 0
\end{array}\right]$. Similarly,
for $V_5$ in the same example, $C^*_{PL}(X,V_5)$ is quasi-isomorphic to the complex
$((\MM\otimes V_5^{ss})^\pi, D')$ with the differential
\[
D'(\omega\otimes {}^t[k_1\ k_2\ k_3])
=d\omega \otimes {}^t[k_1\ k_2\ k_3]
-(e_1/c_1)s_1\omega\otimes {}^t[k_2\ 0\ 0]-(e_2/c_2)s_2\omega\otimes {}^t[k_3\ 0\ 0].
\]
In principle, we can get a similar algebraic complex  for any finite dimensional representation of $\pi$ by successive applications of Lemmas \ref{Lextension_isom} and  \ref{Lrealization} if we can find a polynomial form as in Lemma \ref{Lrealization}.\\ 
\indent In the proof of Theorem \ref{Tmain}, we obtained a semi-simple $\pi$-CDGA which is quasi-isomorphic to the equivariant minimal model in terms of a local system model. A similar procedure works for a general space with the fundamental group $\pi$. Let $E$ be a connected space with $\pi_1(E)=\pi$ and $\pi_i(E)\otimes_{\ZZ} \QQ$ finite dimensional for each $i\geq 2$. We may assume the existence of a 2-connected fibration $p:E\to T^2$ by replacing $E$ with a weak homotopy equivalent space if necessary. Let $B$ be a local system model of $E$. By definition, $B$ is connected with $\mathcal{F}(E,p)$ by a zigzag of quasi-isomorphisms of $A_{T^2}$-algebras. Let $pt \in T^2$ be the 0-cell and $(\Lambda W, d)$ the minimal model of $B_{pt}$ (in the sense of Sullivan). We fix quasi-isomorphisms $\psi_{pt}: (\Lambda W,d)\to B_{pt}$ and $\psi_{\sigma_i}:(\Lambda W, d)\to B_{\sigma_i}$ of CDGA's for $i=1,2$. We can take maps of CDGA's $\bar d_{ij}: (\Lambda W,d)\to (\Lambda W, d)$ such that the square
\[
\xymatrix{
(\Lambda W, d)\ar[r]^{\bar d_{ij}}\ar[d]^{\psi_{\sigma_i}} & (\Lambda W, d)\ar[d]^{\psi_{pt}} \\
B_{\sigma_i}\ar[r]^{d_j} & B_{pt}}
\] 
commutes up to homotopy for $i=1,2,\  j=0,1$. Since $(\Lambda W,d)$ is minimal, $\bar d_{ij}$ is an isomorphism. We regard $W$ as the representation of  $\pi$ on which $g_i$ acts as the map  on the indecomposables induced  by $\bar d_{i0}\circ (\bar d_{i1})^{-1}$. Suppose $\kk$ is algebraically closed. We can take a basis $\{x_j\}_{j\in J}$ of the semi-simplification of $W$ such that $g_i$ acts by a scalar $c_{ij}\in \kk$ on $x_j$. Let $M_j$ be the one dimensional representation of $\pi$ on which $g_i$ acts by $c_{ij}^{-1}$. Let $S$ be the set of collections $(k_j)_{j\in J}$ of non-negative integers such that for (possibly) only  finitely many $j$'s, $k_j\not=0$ and  the representations $\otimes_{j\in J}M_j^{k_j}$ with $(k_j)\in S$ appear exactly once in each isomorphism class of representations of the form $\otimes_{j\in J}M_j^{k'_j}$ with $k'_j\geq 0$ for any $j\in J$ and $k'_j\not =0$ only for  finite $j$'s. Let $\tilde B$ be the semi-simple $\pi$-CDGA defined as follows. As a complex, set
\[
\tilde B=\bigoplus_{(k_j)\in S}\Gamma(B\otimes (\otimes_{j\in J}M_j^{k_j})),
\]
where $M_j$ is regarded as a local system on $T^2$. The product on $\tilde B$ is defined similarly to that of  $\tilde \LL$ in the proof of Theorem \ref{Tmain} by fixing a generator of $M_j$ and  using the product of $B$ and the tensors of $M_j$'s. The generator $g_i$ acts on $\Gamma(B\otimes (\otimes_{j\in J}M_j^{k_j}))$ by $\prod_{j\in J}c_{ij}^{k_j}$. The proof of the following theorem is completely similar to Theorem \ref{Tmain}.
\begin{thm}\label{Llocal_equiv}
Suppose  $\kk$ is algebraically closed. Let $E$ be a pointed connected space with $\pi_1(E)=\pi$ and $\pi_i(E)\otimes_{\ZZ} \QQ$ finite dimensional for each $i\geq 2$ and $B$ a local system model of $E$. With the above notations, $\tilde B$ is connected with  the equivariant minimal model of $E$ by a zigzag of quasi-isomorphisms between semi-simple $\pi$-CDGA's.  \hfill \qedsymbol
\end{thm}
Clearly, we have a similar claim for the case of $\pi_1(E)=\ZZ^n$ with $n\geq 3$.

\end{document}